\documentclass[a4paper,11pt]{amsart}
\usepackage{amsfonts}
\usepackage{amsthm}
\usepackage{amssymb}
\usepackage{amsmath}
\usepackage{enumerate}
\usepackage{url}
\usepackage{color}
\usepackage{comment}
\usepackage{a4wide}

\begin{document}
\title[A pair correlation problem and counting lattice points]{A pair correlation problem, and counting lattice points with the zeta function}
\author[C.\ Aistleitner, D.\ El-Baz, M.\ Munsch]{Christoph Aistleitner, Daniel El-Baz, Marc Munsch}

\subjclass[2010]{Primary 11K06, 11J83, 11M06; Secondary 11B05, 11J25, 11J71 }
 \keywords{Pair correlation, Riemann zeta function, lattice points, Diophantine inequality.} 

\newcommand{\mods}[1]{\,(\mathrm{mod}\,{#1})}

\dedicatory{Dedicated to the memory of Professor Aleksandar Ivi\'c (1949-2020)}

\begin{abstract}
The pair correlation is a localized statistic for sequences in the unit interval. Pseudo-random behavior with respect to this statistic is called Poissonian behavior. The metric theory of pair correlations of sequences of the form $(a_n \alpha)_{n \geq 1}$ has been pioneered by Rudnick, Sarnak and Zaharescu. Here $\alpha$ is a real parameter, and $(a_n)_{n \geq 1}$ is an integer sequence, often of arithmetic origin. Recently, a general framework was developed which gives criteria for Poissonian pair correlation of such sequences for almost every real number $\alpha$, in terms of the additive energy of the integer sequence $(a_n)_{n \geq 1}$. In the present paper we develop a similar framework for the case when $(a_n)_{n \geq 1}$ is a sequence of reals rather than integers, thereby pursuing a line of research which was recently initiated by Rudnick and Technau. As an application of our method, we prove that for every real number $\theta>1$, the sequence $(n^\theta \alpha)_{n \geq 1}$ has Poissonian pair correlation for almost all $\alpha \in \mathbb{R}$.
\end{abstract}

\maketitle

\newtheorem{cor}{Corollary}
\newtheorem{prop}{Proposition}
\newtheorem{claim}{Claim}
\newtheorem{lemma}{Lemma}
\newtheorem{thm}{Theorem}
\newtheorem{defn}{Definition}
\newtheorem{conj}{Conjecture}
\newcommand{\tmop}[1]{\ensuremath{\operatorname{#1}}}
\theoremstyle{definition}
\newtheorem{exmp}{Example}

\theoremstyle{remark}
\newtheorem{rmk}{Remark}

\newcommand{\ve}{\varepsilon}
\newcommand{\R}{\mathbb{R}}
\newcommand{\dmu}{~d\mu(\alpha)}

\allowdisplaybreaks
\parindent0mm

\section{Introduction and statement of results}

A sequence $(y_n)_{n \geq 1}$ of real numbers is called uniformly distributed  (or equidistributed) modulo one if for all intervals $A \subset [0,1)$ the asymptotic equality 
\begin{equation} \label{ud}
\lim_{N \to \infty} \frac{1}{N} \sum_{n=1}^N \mathbf{1}_A (y_n) = \lambda(A) 
\end{equation}
holds. Here $\mathbf{1}_A$ is the indicator function of $A$, extended periodically with period 1, and $\lambda$ denotes Lebesgue measure. Uniform distribution theory has a long history, going back to the seminal paper of Hermann Weyl \cite{weyl}. For general background, see \cite{dts,kn}. Uniform distribution of a sequence can be seen as a pseudo-randomness property, in the sense that a sequence $(Y_n)_{n \geq 1}$ of independent, identically distributed random variables having uniform distribution on $[0,1)$ satisfies \eqref{ud} almost surely as a consequence of the Glivenko--Cantelli theorem; thus a deterministic sequence $(y_n)_{n \geq 1}$ which is uniformly distributed mod 1 exhibits the same behavior as a typical realization of a random sequence.\\

A sequence $(y_n)_{n \geq 1}$ is said to have Poissonian pair correlation if for all real numbers $s \geq 0$,
\[
\lim_{N \to \infty} \frac{1}{N} \sum_{\substack{1 \leq m,n \leq N,\\m \neq n}} \mathbf{1}_{[-s/N,s/N]} (y_n - y_m) = 2 s.
\]
This notion is motivated by questions from theoretical physics, and plays a key role in the Berry--Tabor conjecture; see \cite{marklof} for more information. Just like equidistribution, Poissonian pair correlation can be seen as a pseudo-randomness property, since a random sequence $(Y_n)_{n \geq 1}$ as above almost surely has Poissonian pair correlation. However, clearly the two properties are of a rather different nature. While equidistribution is a ``large-scale'' statistic (where the test interval always remains the same), the pair correlation is a highly localized statistic (where the size of the test interval shrinks in proportion with $N$). Note that the two properties are not independent: it is known that a sequence having Poissonian pair correlation necessarily must be equidistributed \cite{alp,gl,mark}, whereas the opposite implication is generally false. An illustrative example is the sequence $(n \alpha)_{n \geq 1}$, which is equidistributed if and only if $\alpha \not\in \mathbb{Q}$, but which fails to have Poissonian pair correlation for any $\alpha$ (see \cite{LarcherStockinger} for a more general result along those lines). \\

The theory of uniform distribution modulo one can be said to be relatively well understood (at least in the one-dimensional case). Many specific sequences are known which are uniformly distributed mod one. In contrast, only very few specific results are known in the pair correlation setting. A notable exception is the sequence $(\sqrt{n})_{n \in \mathbb{Z}_{\geq 1} \setminus \Box}$, which is known to have Poissonian pair correlation \cite{emv}. The sequence $(n^2 \alpha)_{n \geq 1}$ is conjectured to have Poissonian pair correlation under mild Diophantine assumptions on $\alpha$, but only partial results are known in this direction \cite{hbpc,my, rsz,tru}. Lacking specific examples, it is natural to turn to a metric theory instead. Let $(a_n)_{n \geq 1}$ be a sequence of distinct integers, let $\alpha \in \mathbb{R}$, and consider sequences of the form $(a_n \alpha)_{n \geq 1}$. The metric theory of such sequences with respect to equidistribution is very simple: for every such $(a_n)_n$, the sequence $(a_n \alpha)_n$ is uniformly distributed mod 1 for almost all $\alpha$ \cite{weyl}. The situation with respect to pair correlation is much more delicate. Pioneering work in this area was carried out by Rudnick, Sarnak and Zaharescu \cite{rsarnak,rz}. As noted above, $(n \alpha)_{n \geq 1}$ does not have Poissonian pair correlation for any $\alpha$. However, for any polynomial $p \in \mathbb{Z}[X]$ of degree at least 2, the pair correlation of $(p(n)\alpha)_n$ is Poissonian for almost all $\alpha$. For related results, see for example \cite{bpt,clz,rz_lac}.\\

Recently, a simple criterion was established in \cite{all} which allows to decide whether the sequence $(a_n \alpha)_n$ has Poissonian pair correlation for almost all $\alpha$ for many naturally arising integer sequences $(a_n)_n$ . Let $E_N$ denote the number of solutions $(n_1,n_2,n_3,n_4)$ of the equation
\begin{equation} \label{dio}
a_{n_1} - a_{n_2} + a_{n_3} - a_{n_4} = 0,
\end{equation}
subject to $1 \leq n_1,n_2,n_3,n_4 \leq N$. This quantity is called the \emph{additive energy} in the additive combinatorics literature (see \cite{gry,TaoVu}). Note that $N^2 \leq E_N \leq N^3$ for every $(a_n)_n$ and every $N$. The criterion is as follows. If a sequence $(a_n)_n$ satisfies $E_N \ll N^{3-\ve}$ for some $\ve>0$, then $(a_n \alpha)_n$ has Poissonian pair correlation for almost all $\alpha$. If in contrast $E_N \gg N^3$, then the conclusion fails to be true. For further refinements of this criterion, and for remaining open problems, see \cite{alt,bcgw,bw,ls}. We emphasize that all that was written in this paragraph requires $(a_n)_n$ to be a sequence of \emph{integers}.\\

Very little is known in the metric theory of pair correlation of sequences $(x_n \alpha)_n$ when $(x_n)_n$ is a sequence of reals rather than integers. One step in this general direction is \cite{clz}, where $(x_n)_n$ is allowed to take rational values and the results obtained depend on the size of the denominators of these rationals. A general result was obtained recently in \cite{rt}, where the authors gave a criterion formulated in terms of the number of solutions of a certain Diophantine inequality. The criterion is as follows: for a sequence $(x_n)_n$, assume that there exist some $\ve>0$ and $\delta>0$ such that the number of integer solutions $(n_1,n_2,n_3,n_4,j_1,j_2)$ of the equation
\begin{equation} \label{dioph_in}
\left| j_1 (x_{n_1} - x_{n_2}) - j_2  ( x_{n_3} - x_{n_4} ) \right| < N^\ve, 
\end{equation}
subject to $1 \leq |j_1|,|j_2| \leq N^{1+\ve},~1 \leq n_1,n_2,n_3,n_4 \leq N,~n_1 \neq n_2, ~n_3 \neq n_4$, is of order $\ll N^{4 - \delta}$, then $(x_n \alpha)_n$ has Poissonian pair correlation for almost all $\alpha$. It is verified in \cite{rt} that this condition is satisfied for lacunary sequences. A condition in the spirit of \eqref{dioph_in} arises very naturally when studying this sort of problem (cf.\ also \cite{rz}); we will encounter a variant of this condition in Equation \eqref{line_1} below. In particular, it is very natural that in the integer case one has to count solutions of Diophantine \emph{equations}, while in the real-number setting one has to count solutions of Diophantine \emph{inequalities}. The problem with \eqref{dioph_in} is that it is in general rather difficult to verify whether this condition is satisfied for a given sequence or not, with issues being caused in particular by the presence of the coefficients $j_1$ and $j_2$. The purpose of the present paper is to give a simplified criterion, in the spirit of the criterion of \cite{all} which was specified in terms of the number of solutions of the equation \eqref{dio}.\\

\begin{thm} \label{th1}
Let $(x_n)_{n \geq 1}$ be a sequence of positive real numbers for which there exists a constant $c>0$ such that $x_{n+1} - x_n \geq c,~n \geq 1$. Let $E_N^*$ denote the number of solutions $(n_1, n_2, n_3, n_4)$ of the inequality
\begin{equation} \label{dioph_in_1}
|x_{n_1} - x_{n_2} + x_{n_3} - x_{n_4}| <1,
\end{equation}
subject to $n_i \leq N, ~i = 1,2,3,4$. Assume that there exists some $\delta>0$ such that $E_N^* \ll N^{183/76-\delta}$ as $N \to \infty$. Then the sequence $(x_n \alpha)_{n \geq 1}$ has Poissonian pair correlation for almost all $\alpha \in \mathbb{R}$. \\
\end{thm}

The exponent $183/76 \approx 2.408$ in the conclusion of the theorem comes from a bound for the $178/13$-th moment of the Riemann zeta function on the critical line due to Ivic \cite{Ivic}, building on earlier work of Heath-Brown \cite{HB}. Conditionally under the Lindel\"of hypothesis, our bound for $E_N^*$ can be relaxed to $E_N^* \ll N^{3 - \ve}$ for any $\ve>0$, which would be in accordance with the results known for the integer case.\\

Theorem \ref{th1} applies, for example, to all sequences of the form $x_n = p(n),~n \geq 1$, where $p$ is a quadratic polynomial with real coefficients. For such a sequence $(x_n)_n$ we have $E_N^* \ll N^{2 + \ve}$ for any $\ve>0$ by Lemma 5.2 of \cite{bkw}. Theorem \ref{th1} also applies to $x_n = p(n)$ for every polynomial $p \in \mathbb{R}[X]$ of degree $d \geq 3$, under the additional assumption that the coefficient of $x^{d-1}$ is rational\footnote{This could be relaxed to assuming some Diophantine condition on this coefficient.}; the required bound for $E_N^*$ then follows, after eliminating this coefficient, from Lemma \ref{rs_lemma} below (with the choice of $\theta = d$ and $\gamma = N^{d-2}$). The extra assumption on the second coefficient is most likely redundant, but we have not been able to establish the necessary bound for $E_N^*$ without it. A famous open conjecture in additive combinatorics asserts that $E_N \ll N^{2 + \ve}$ for all convex sequences $(x_n)_n$, which would provide many further applications of our theorem; however, unfortunately the best current bound in this direction (Shkredov's $32/13 \approx 2.46$ from \cite{shkr}) is just beyond the range of applicability of our theorem.\\

Bounding the number of solutions of \eqref{dioph_in_1} is necessary to control the variance of the pair correlation function. When carefully reading the proof of Theorem \ref{th1} it becomes visible that not all solutions of \eqref{dioph_in_1} contribute equally to the variance, but that rather a 4-tuple $(n_1,n_2,n_3,n_4)$ with $x_{n_1} - x_{n_2} + x_{n_3} - x_{n_4} = \gamma$ for some $\gamma \in (-1,1)$ has a stronger effect on the variance the smaller the absolute value of $\gamma$ is. This suggests to consider the quantity $E_{N,\gamma}^*$, which is defined as the number of solutions $(n_1, n_2, n_3, n_4)$ of the inequality
$$
|x_{n_1} - x_{n_2} + x_{n_3} - x_{n_4}| < \gamma,
$$
for $\gamma \in (0,1]$ and subject to $n_i \leq N, ~i = 1,2,3,4$. Very informally speaking, one might expect that $E_{N,\gamma}^*$ scales as $E_{N,\gamma}^* \approx \gamma E_N^*$ for a ``randomly behaved'' real sequence $(x_n)_n$, except for the contribution of the trivial solutions $n_1=n_2$ and $n_3 =n_4$ which always is of order $N^2$. The following theorem states that being able to control $E_{N,\gamma}^*$ as a function of $\gamma$ indeed allows us to deduce metric pair correlations in some cases where the condition on the additive energy in Theorem \ref{th1} fails to hold.\\

\begin{thm} \label{th2_new}
Let $(x_n)_{n \geq 1}$ be a sequence of positive real numbers for which there exists $c>0$ such that $x_{n+1} - x_n \geq c,~n \geq 1$. Assume that there exists some $\delta>0$ such that for all $\eta>0$ we have
\begin{equation} \label{engamma_ass}
E_{N,\gamma}^* \ll_{\eta,\delta} N^{2 + \eta} + \gamma N^{3-\delta}
\end{equation}
as $N \to \infty$, uniformly for $\gamma \in (0,1]$. Then  the sequence $(x_n \alpha)_{n \geq 1}$ has Poissonian pair correlation for almost all $\alpha \in \mathbb{R}$. \\
\end{thm}

Theorem \ref{th2_new} is tailor-made for an application to the sequence $x_n = n^\theta$. For that sequence, equation \eqref{engamma_ass} holds provided $\theta > 1$ as a consequence of Lemma \ref{rs_lemma} below. Since that particular problem was a key driving force for writing the present paper, we formulate this conclusion as a theorem rather than just as a corollary.\\

\begin{thm} \label{th2}
For every real number $\theta>1$ and almost every $\alpha \in \mathbb{R}$, the sequence $(n^\theta \alpha)_{n \geq 1}$ has Poissonian pair correlation.\\
\end{thm}

As noted above, the conclusion of Theorem \ref{th2} is not true when $\theta = 1$. It seems plausible that the conclusion of the theorem is valid again for $0 < \theta < 1$. However, this cannot be proved with the methods used in the present paper, which break down in the case of a sequence $(x_n)_n$ whose order of growth is only linear or even slower. We will address this aspect at the very end of the paper, where we also formulate some further open problems.\\

In conclusion we note that Technau and Yesha recently obtained a result which is somewhat similar to our Theorem \ref{th2}, but which is ``metric'' in the exponent rather than in a multiplicative parameter. More precisely, they showed that $(n^\theta)_n$ has Poissonian pair correlation for almost all $\theta > 7$. Their paper also contains similar results on higher correlations, which require a larger value of $\theta$. From a technical perspective, their problem is rather different from ours. For details see their paper \cite{ty}.

\section{Preliminaries}

As in the introduction, let $\mathbf{1}_{[-s/N,s/N]} (x)$ denote the indicator function of the interval $[-s/N,s/N]$, extended with period 1. That is, 
$$
\mathbf{1}_{[-s/N,s/N]} (x) =  \left\{ \begin{array}{ll} 1 & \text{if $x - \langle x \rangle \in [-s/N,s/N]$},\\0 & \text{otherwise,} \end{array} \right.
$$
where $\langle x \rangle$ denotes the nearest integer to $x$. We wish to show that under the assumptions of Theorem \ref{th1} we have for almost all $\alpha \in \mathbb{R}$
\begin{equation} \label{conv_ind}
\frac{1}{N} \sum_{\substack{1 \leq m, n \leq N, \\ m \neq n}} \mathbf{1}_{[-s/N,s/N]} (x_m \alpha - x_n \alpha) \to 2s 
\end{equation}
as $N \to \infty$ for all $s \geq 0$. It is well-known that for any $s$ and $N$, and for any positive integer $K$ there exist trigonometric polynomials $f_{K,s,N}^+(x)$ and $f_{K,s,N}^-(x)$ of degree at most $K$ such that 
\begin{equation} \label{f1_app}
f_{K,s,N}^-(x) \leq \mathbf{1}_{[-s/N,s/N]} (x) \leq f_{K,s,N}^+(x)
\end{equation}
for all $x$, and such that
\begin{equation} \label{f2_app}
\int_0^1 f_{K,s,N}^\pm (x) ~dx = 2s/N \pm \frac{1}{K+1}.
\end{equation}
Furthermore, the $j$-th Fourier coefficient $c_j$ of $f_{K,s,N}^-$ satisfies
\begin{equation} \label{Four-c}
|c_j| \leq \min \left(\frac{2s}{N},\frac{1}{\pi |j|} \right) + \frac{1}{K+1}
\end{equation}
for all $j$, and an analogous bound holds for the Fourier coefficients of $f_{K,s,N}^+$. These trigonometric polynomials are called Selberg polynomials, and their construction is described in detail in Chapter 1 of \cite{mont}.\\

Instead of establishing the required convergence relation \eqref{conv_ind} for indicator functions, we will rather work with the trigonometric polynomials $f_{K,s,N}^+$ and $f_{K,s,N}^-$ instead, which is technically more convenient. More precisely, in order to obtain \eqref{conv_ind} it suffices to prove the following. For every fixed positive integer $r$, and for every fixed real number $s \geq 0$, we have
\begin{equation} \label{conv_f}
\frac{1}{N} \sum_{\substack{1 \leq m, n \leq N, \\ m \neq n}} f_{rN,s,N}^+ (x_m \alpha - x_n \alpha) \sim N \int_0^1 f^+_{rN,s,N}(x) dx
\end{equation}
as $N \to \infty$, for almost all $\alpha \in \mathbb{R}$, and the same is true when $f^+$ is replaced by $f^-$. The desired result for indicator functions then follows from \eqref{f1_app} and \eqref{f2_app} and letting $r \to \infty$ (see \cite{rsarnak, rt} for more details). 

To establish \eqref{conv_f} we prove that the ``expected value'' (with respect to $\alpha$) of the left-hand side is asymptotic to the right-hand side, and that the ``variance'' of the left-hand side of \eqref{conv_f} is not too large. An application of Chebyshev's inequality together with the Borel--Cantelli lemma then gives the desired result. As usual in such problems, controlling the expectation is easier than controlling the variance. \\

We will obtain the required bound for the expectation in Section \ref{sec_av}, and the bound for the variance in Sections \ref{sec_var} and \ref{sec_zeta}. In Section \ref{sec_end} we conclude the proof of Theorem \ref{th1}. Section \ref{sec_new} contains all of the necessary modifications for the proof of Theorem \ref{th2_new}, and in Section \ref{sec_theta} we show that the sequence $(n^\theta)_n$ indeed allows an application of Theorem \ref{th2_new}. Finally, in Section \ref{sec_close} we discuss limitations of our method, and outline open problems and directions for future research.

\section{Proof of Theorem \ref{th1}: Controlling expectations} \label{sec_av}

Throughout the argument we assume that a positive integer $r$ and a positive real $s$ are fixed. We write $f_N$ for the function $f^+_{rN,s,N}$, as defined in the previous section (or for the function $f^-_{rN,s,N}$ --- both cases work in exactly the same way). We want to control the ``expected value'' with respect to $\alpha$ of the left-hand side of \eqref{conv_f} as $N \to \infty$. In the case when $(x_n)_{n \geq 1}$ is an integer sequence everything is periodic with period 1, and it is appropriate to integrate over $\alpha \in [0,1]$ with respect to the Lebesgue measure. In our case, when $(x_n)_{n \geq 1}$ is a sequence of reals, we do not have such periodicity. We thus have to integrate over all $\alpha \in \mathbb{R}$ with respect to an appropriate measure $\mu$, which is absolutely continuous with respect to the Lebesgue measure, so that a $\mu$-almost everywhere conclusion implies a Lebesgue-almost everywhere conclusion. A good choice for the measure $\mu$ is the measure whose density with respect to the Lebesgue measure is given by
\begin{equation} \label{measure}
d\mu(x) = \frac{2 (\sin (x/2))^2}{\pi x^2} dx. 
\end{equation}
The Fourier transform of $x \mapsto \frac{2 (\sin (x/2))^2}{\pi x^2}$ is a non-negative real function which is supported on the interval $(-1,1)$, and which is uniformly bounded by $1/\sqrt{2 \pi}$. Note that the measure $\mu$ is normalized such that $\mu(\mathbb{R})=1$.\\

Expanding the function $f_N(x)$ into a Fourier series
\[
\sum_{j \in \mathbb{Z}} c_j e^{2 \pi i j x},
\]
by construction we have $c_j = 0$ when $|j| > rN$, and $|c_j| \leq 2s/N + 1/(rN) \ll N^{-1}$ for all $j$ (recall that $r$ and $s$ are assumed to be fixed). Moreover we have $c_0 = \int_0^1 f_{N}(x) dx$. Using the fact that the Fourier transform of the measure $\mu$ is supported on $(-1,1)$ and uniformly bounded, we obtain
\begin{align}
& \left\vert\int_\R \frac{1}{N} \sum_{\substack{1 \leq m, n \leq N, \\ m \neq n}} f_N (x_m \alpha - x_n \alpha) \dmu - N \int_0^1 f_{N}(x) dx \right\vert \nonumber\\
& \ll  \underbrace{\left(N -\frac{N(N-1)}{N} \right)}_{= 1} \underbrace{\int_0^1 f_{N}(x) dx}_{\ll N^{-1}}  \nonumber\\ 
&   \quad + \frac{1}{N} \sum_{1 \leq \vert j\vert \leq r N} |c_{j}| \left| \int_\R \sum_{\substack{1 \leq m, n \leq N, \\ m \neq n}} e^{2 \pi i j (x_m \alpha - x_n \alpha)} \dmu \right| \nonumber\\
& \ll  N^{-1} + N^{-2} \underbrace{\sum_{1 \leq \vert j\vert \leq r N} \sum_{\substack{1 \leq m, n \leq N, \\ m \neq n}} \mathbf{1} \big(|j (x_m - x_n)| < 1 \big)}_{\ll N \text{ due to the growth assumption on $(x_n)_{n \geq 1}$}} \nonumber \\
& \ll N^{-1},
\end{align} 
where we estimated $|c_j|$ using \eqref{Four-c}. Thus we have
\begin{equation}\label{expect}
\int_\R \frac{1}{N} \sum_{\substack{1 \leq m, n \leq N, \\ m \neq n}} f_N (x_m \alpha - x_n \alpha) \dmu = N \int_0^1 f_{N}(x) dx + O(1/N),
\end{equation} 
as desired. 

Controlling the variances is more difficult, and will be done in the next two sections.

\section{Proof of Theorem \ref{th1}: Controlling variances} \label{sec_var}

We keep the setup as in Section \ref{sec_av} above, that is, we assume that $r$ and $s$ are fixed, and we write $f_N$ for either $f^+_{rN,s,N}$ or $f^-_{rN,s,N}$. Furthermore, we write $h_N$ for the centered version of $f_N$, that is, for the function
\begin{equation} \label{gdef}
h_N(x) = f_N(x) - \int_0^1 f(x) ~dx = \sum_{\substack{j \in \mathbb{Z},\\j \neq 0}} c_j e^{2 \pi i j x}.
\end{equation}
We wish to estimate the ``variance'' of our localized counting function, or more precisely the quantity
\begin{equation} \label{wish_to}
\mathrm{Var}(h_N,\mu):=\int_\R \left(\frac{1}{N} \sum_{\substack{1 \leq m, n \leq N, \\ m \neq n}} h_N (x_m \alpha - x_n \alpha)\right)^2 \dmu. 
\end{equation} 
The following bound on $\mathrm{Var}(h_N,\mu)$ is the crucial ingredient in our proof of Theorem \ref{th1}.
\begin{lemma}\label{lem:var_bound}
For every $\ve >0$ we have, as $N \to \infty$,
\begin{equation*} 
    \mathrm{Var}(h_N, \mu) \ll \max \left(N^{-\ve/8} + N^{-183/89 + 3\ve} (E_N^*)^{76/89}, E_N^* N^{-2.49+ 4\ve} \right).
\end{equation*}
\end{lemma}

For the convenience of the reader, we note at this point that our assumption that there is some $\delta > 0$ such that $E_N^* \ll N^{183/76 - \delta}$ ensures together with Lemma \ref{lem:var_bound} that there is some $\delta' > 0$ such that $\mathrm{Var}(h_N, \mu) \ll N^{-\delta'}$, which is sufficient to deduce Theorem \ref{th1} (see Section \ref{sec_end} for details). We also note that conditionally under the Lindel\"of hypothesis the bound which follows from our method is $\mathrm{Var}(h_N, \mu) \ll N^{-3+\ve} E_N^*$.

\section{Proof of Lemma \ref{lem:var_bound}: Lattice point counting via the Riemann zeta function} \label{sec_zeta}
\subsection{A first reduction} Squaring out in \eqref{wish_to} and using again the properties of the Fourier transform of the measure $\mu$, we can bound $\mathrm{Var}(h_N, \mu) $ by 
\begin{align*}
&\int_\R \frac{1}{N^2} ~\sum_{\substack{1 \leq n_1, n_2, n_3, n_4 \leq N, \\ n_1 \neq n_2, ~n_3 \neq n_4}} ~\sum_{\substack{j_1,j_2 \in \mathbb{Z},~j_1, j_2 \neq 0\\|j_1|,|j_2| \leq r N}} ~\underbrace{|c_{j_1} c_{j_2}|}_{\ll N^{-2}} e^{2 \pi i \alpha (j_1 (x_{n_1} - x_{n_2}) - j_2 (x_{n_3} - x_{n_4}))} \dmu \nonumber\\
& \ll \frac{1}{N^4} ~\sum_{\substack{1 \leq n_1, n_2, n_3, n_4 \leq N, \\ n_1  > n_2, ~n_3 > n_4}} ~\sum_{\substack{1 \leq j_1,j_2 \leq r N}}~ \mathbf{1} \big(|j_1 (x_{n_1} - x_{n_2}) - j_2 (x_{n_3} - x_{n_4})| < 1 \big), \nonumber
\end{align*}
thereby essentially arriving at \eqref{dioph_in}. For technical reasons, in this paper we prefer to localize the variables $j_1,j_2$ into dyadic regions and thus apply the Cauchy--Schwarz inequality to \eqref{wish_to}. To simplify later formulas we also replace the differences $x_{n_1} - x_{n_2}$ and $x_{n_3} - x_{n_4}$ by their respective absolute values using the parity of $h_N$. Then, writing $U$ for the smallest integer for which $2^U \geq r N$, we can bound $\mathrm{Var}(h_N, \mu)$ by
\begin{align}
 &  \int_{\mathbb{R}} \left(\frac{1}{N}~ \sum_{u=1}^U  ~\sum_{\substack{1 \leq m, n \leq N, \\ m \neq n}} ~\sum_{ 2^{u-1}  \leq |j| < 2^u} c_j e^{2 \pi i j |x_m - x_n| \alpha} \right)^2 \dmu   \nonumber\\
& \ll \frac{1}{N^2} \int_{\mathbb{R}} \left(\sum_{k=1}^U 1\right) \sum_{u=1}^{U}\left\vert \sum_{\substack{1 \leq m, n \leq N, \\ m \neq n}}~ \sum_{2^{u-1}  \leq |j| < 2^u} c_j e^{2 \pi i j |x_m - x_n| \alpha} \right\vert^2 \dmu \nonumber\\
& \ll  \frac{\log N}{N^4} ~\sum_{u=1}^U ~\sum_{\substack{1 \leq n_1, n_2, n_3, n_4 \leq N, \\ n_1 \neq n_2, ~n_3 \neq n_4}} ~\sum_{\substack{2^{u-1} \leq j_1,j_2 < 2^u}} ~\mathbf{1} \left( \Big|j_1 |x_{n_1} - x_{n_2}| - j_2 |x_{n_3} - x_{n_4}| \Big| < 1 \right). \label{line_1}
\end{align}
Thus we have reduced the problem of estimating the variance to a problem of bounding the number of solutions of a Diophantine inequality. 

\subsection{Counting solutions by using the Riemann zeta function} We will relate the counting problem in Equation \eqref{line_1} to the problem of bounding a twisted moment of the Riemann zeta function. Before we return to the proof, we point out the difference between the real-number case (in this paper) and the corresponding results for the case of $(x_n)_{n \geq 1}$ being an integer sequence. In the integer case, the problem of estimating the variance of the pair correlation function can be reduced to counting solutions of $j_1 (x_{n_1} - x_{n_2}) = j_2 (x_{n_3} - x_{n_4})$. Note that this is in accordance with the situation in the present paper, where we count solutions to $|j_1 (x_{n_1} - x_{n_2}) - j_2 (x_{n_3} - x_{n_4})|<1$, with the difference that in the integer case ``$<1$'' implies ``$=0$''. The number of solutions of the counting problem in the integer case is essentially governed by what is called a GCD sum. It is known that such sums have a connection with the Riemann zeta function (see \cite{a1,hilb}), and strong estimates for such sums were obtained in \cite{abs,bs,dlbT}. Our approach below is motivated by a beautiful argument of Lewko and Radziwi{\l}{\l} \cite{lr}, who showed how the relevant GCD sum can be estimated in terms of a twisted moment of a random model of the Riemann zeta function on the critical line.\footnote{See also \cite{dlbmt} for links between twisted moments of character sums and GCD sums. Furthermore, see \cite{shkr_new} for a very recent paper of Shkredov, where he applies GCD sums and methods from \cite{lr} to give upper bounds for the maximal length of arithmetic progressions contained in sets with small product set.} The randomization was crucial in their argument for different reasons, one being that the required distributional estimates for extreme values of the actual Riemann zeta function are not known unconditionally. Their argument relied crucially on the fundamental theorem of arithmetic, and thus on the fact that they were dealing with integer sequences. In the real-number case the situation is much more delicate. We will relate our counting problem to a convolution formula for the Riemann zeta function. The kernel will be chosen for its good properties with respect to the Fourier transform (positivity and localized support) which allow to overcount without substantial loss. To summarize, in our argument below we will use of a combination of ideas from \cite{a1,bs,BSconv,dlbT} and \cite{lr}.\\

Let $(x_n)_{n \geq 1}$ be the sequence from the statement of Theorem \ref{th1}. Let $M = N^2-N$, and let $\{z_1, \dots, z_M\}$ be the multi-set of all absolute differences $\{|x_m - x_n|:~1 \leq m, n \leq N,~m \neq n\}$, meaning that we allow repetitions in the definition. For a given positive integer $u$ with $2^u \leq 2 r N$, we wish to estimate
\begin{equation} \label{to_sum}
\sum_{1 \leq m,n \leq M} ~\sum_{\substack{2^{u-1} \leq j_1,j_2 < 2^u}} ~\mathbf{1} \big(|j_1 z_m - j_2 z_n| < 1 \big).
\end{equation}
We write $\zeta(\sigma+it)$ for the Riemann zeta function.
We also write $\Phi(t)=e^{-t^2/2}$, and note that this function has a positive Fourier transform given by $\widehat{\Phi}= \sqrt{2\pi} \Phi$.
Throughout the proof $\ve>0$ is a small constant, and we take it for granted that $N$ is ``large''.\\

Our argument proceeds by splitting into two cases depending on the size of $\min\{z_m,z_n\}$. We first treat the case when $z_m,z_n$ are both at least of size $N^{1.01}$. We then treat the case when one of $z_m$ or $z_n$ is ``small'', which because of our dyadic splitting essentially amounts to saying that both variables are small.\\

$\bullet$ {\bf Case 1:~Counting solutions for $z_m,z_n \geq N^{1.01}$.} \\

Let $u$ be given such that $2^{u-1} \leq j_1,j_2 \leq 2^u$. Set $T = 2^u N^{1+\ve/2}$. For any integer $k \geq N^{1.01}$, we set
\begin{equation} \label{b_k_def}
b_k = \sum_{m=1}^M \mathbf{1} \big(z_m \in [k, k+1) \big),
\end{equation} 
while for $k < N^{1.01}$ we set $b_k=0$. Clearly we have
$$
\sum_{k=1}^\infty b_k \leq M = N^2 - N.
$$
We split the interval $[1, \infty)$ into a disjoint union of intervals $I_h$ for $h \geq 0$, where
$$
I_h = \left[ \left\lceil \left(1 + \frac{1}{T} \right)^h \right\rceil , \left\lceil \left(1 + \frac{1}{T} \right)^{h+1} \right\rceil \right),
$$
and we set 
\begin{equation} \label{ah}
a_h = \left( \sum_{\substack{k \in I_h}} b_k^2 \right)^{1/2}, \qquad h \geq 0.
\end{equation}

Finally, we define a function
\begin{equation} \label{P_def}
P(t) = \sum_{h=0}^\infty a_h \left(1 + \frac{1}{T} \right)^{iht}.
\end{equation}

The following four lemmas correspond to the key steps in our Case 1 analysis.

\begin{lemma}[Controlling the square-integral of $P$ in terms of the additive energy] \label{lemma_new_1} We have
\begin{equation}
\int_\R |P(t)|^2 \Phi(t/T)dt \ll T E_N^*.
\end{equation} 
\end{lemma}

\begin{lemma}[Counting solutions of Diophantine inequalities in terms of $a_h$] \label{lemma_new_2} We have
\begin{align*}
\sum_{2^{u-1} \leq j_1, j_2 < 2^u}  ~\sum_{1 \leq m,n \leq M} ~\mathbf{1} \big(|j_1 z_m - j_2 z_n| < 1 \big) & \ll \sum_{2^{u-1} \leq j_1, j_2 \leq 2^{u}} \sum_{\substack{h_1, h_2 \geq 0, \\ \left|\left(1 + \frac{1}{T} \right)^{h_1-h_2} - \frac{j_2}{j_1} \right| \leq \frac{4}{T}}} a_{h_1} a_{h_2}.
\end{align*}
\end{lemma}

The next step, in Lemma \ref{lemma_new_3}  below, is to relate the sum on the right-hand side of the equation above to a complex integral, where the term $ \sum_{j_1,j_2 \geq 1} (j_1 j_2)^{-1/2} (j_1/j_2)^{it}$ can be roughly interpreted as $|\zeta(1/2+it)|^2$. However, instead of simply using a truncated expression or an approximate functional equation for $\zeta(1/2+it)$, we will rather use of a convolution formula to improve the analysis near $t = 0$. To do so, we introduce the function $K$ defined by
\[
K(u):=\frac {\sin^2( (1+\ve/4) u\log N)}{\pi u^2 (1 + \ve/4) (\log N)},
\]
whose Fourier transform is given by
\[\widehat K(\xi)=\max\left(1-\frac {|\xi|}{2 (1 + \ve/4) \log N}, 0\right).\]
A similar idea was also fruitfully used in a paper of Bondarenko and Seip \cite{BSconv}. The function $K$ is chosen in such a way that we have $\widehat{K} (\log j_1 j_2) \gg  1 - \frac{2 (\log rN)}{2 (1+ \ve/4) \log N} \gg 1$ (where we suppress the dependence on the constants $\ve$ and $r$).\\

\begin{lemma}[Counting solutions of the Diophantine inequality by complex integration] \label{lemma_new_3} 
We have 
 \begin{align}
\sum_{2^{u-1} \leq j_1, j_2 < 2^u} \sum_{\substack{h_1 \geq 0, h_2 \geq 0,\\\left| \left( 1 + \frac{1}{T} \right)^{h_1 - h_2} - \frac{j_2}{j_1} \right| \leq \frac{4}{T}}} a_{h_1} a_{h_2} \ll \frac{2^u}{T} \int_{\mathbb{R}} \sum_{j_1,j_2 \geq 1} \frac{\widehat{K}(\log j_1j_2)}{(j_1 j_2)^{1/2}} \bigg ( \frac{j_1}{j_2} \bigg )^{it} |P(t)|^2  \Phi(t/T) dt. \nonumber
 \end{align} 
\end{lemma}

\begin{lemma}[Estimating the complex integral] \label{lemma_new_4} We have
\begin{equation}
\frac{2^u}{T} \int_{\mathbb{R}} \sum_{j_1,j_2 \geq 1} \frac{\widehat{K}(\log j_1j_2)}{(j_1 j_2)^{1/2}} \bigg ( \frac{j_1}{j_2} \bigg )^{it} |P(t)|^2  \Phi(t/T) dt \ll N^{4-\ve/4} + N^{173/89+ 2\ve} (E_N^*)^{76/89}.
\end{equation}
\end{lemma}

We now prove these four lemmas.

\begin{proof}[Proof of Lemma \ref{lemma_new_1}]
We have
\begin{align}
\int_\R |P(t)|^2 \Phi(t/T)dt & = \int_\R \sum_{h_1,h_2 \geq 0} a_{h_1} a_{h_2} \left(1 + \frac{1}{T} \right)^{(h_1 - h_2) it} \Phi(t/T) dt \nonumber\\
& = T \sum_{h_1,h_2 \geq 0} a_{h_1} a_{h_2} \int_{\mathbb{R}} \exp \left( \left(\log\left(1 + \frac{1}{T} \right) \right) T (h_1-h_2) i y \right) \Phi(y)dy \nonumber\\
& = T \sum_{h_1,h_2 \geq 0} a_{h_1} a_{h_2} \widehat{\Phi} \left( \left( \log\left(1 + \frac{1}{T} \right) \right) T (h_1-h_2) \right) \nonumber\\
& \ll T \sum_{h_1,h_2 \geq 0} a_{h_1} a_{h_2} \widehat{\Phi} \left( \frac{h_1-h_2}{2} \right) \label{orth_a}\\
& \ll T \sum_{h=0}^\infty a_h^2 \label{orth}\\
& \ll T \sum_{k=1}^\infty b_k^2 \ll T \sum_{\substack{m,n=1 \\ \lvert z_m-z_n \rvert <1}}^{M}1  = T E_N^*. \nonumber
\end{align} 
Here we used that $T \log (1+1/T) \geq 1/2$ for sufficiently large $N$ (note that large $N$ implies large $T$), and the Cauchy--Schwarz inequality together with the rapid decay of $\widehat{\Phi}$ to pass from \eqref{orth_a} to \eqref{orth}. We will further comment on the construction of $P(t)$ at the very end of our Case 1 analysis.
\end{proof}

\begin{proof}[Proof of Lemma \ref{lemma_new_2}]
Let $j_1$ and $j_2$ be fixed, and assume without loss of generality that $j_1 \geq j_2$. Let $k \geq N^{1.01}$ be an integer in $I_{h_1}$, and assume that $z_m \in [k,k+1)$. Then the inequality $|j_1 z_m - j_2 z_n| < 1$ is only possible when 
\begin{equation} \label{cons}
\left| \left\lceil\frac{j_1 k}{j_2} \right\rceil - z_n \right| < 4
\end{equation}
(recall that $j_1/j_2 \leq 2$ because $j_1,j_2$ are located in the same dyadic interval). We write $\ell(k) = \lceil j_1 k / j_2 \rceil$. Recall that $j_1 / j_2 \geq 1$ by assumption, so the mapping $k \mapsto \ell(k)$ is injective. Thus we have
\begin{align}
& \sum_{z_m \in I_{h_1}, z_n \in I_{h_2}} \mathbf{1} \big(|j_1 z_m - j_2 z_n| < 1 \big) \nonumber\\
& \ll \sum_{k \in I_{h_1}} ~\sum_{z_m \in [k,k+1)}~ \sum_{\substack{z_n \in I_{h_2},\\ \left|\ell(k)-z_n \right| < 4}} 1 \nonumber\\ 
& \ll \sum_{k \in I_{h_1}} ~\sum_{z_m \in [k,k+1)}~ \sum_{-4 \leq v \leq 3}  ~ \sum_{\substack{z_n \in I_{h_2},\\ z_n \in \left[\ell(k) + v, \ell(k) + v + 1 \right)}} 1 \nonumber\\
& \ll ~ \sum_{-4 \leq v \leq 3} ~\sum_{\substack{k \in I_{h_1} \text{ such that} \\ \ell(k) + v \in I_{h_2}}}~ b_k b_{\ell(k) + v} \nonumber\\
& \ll  \left(\sum_{k \in I_{h_1}} b_k^2  \right)^{1/2} \left(\sum_{\ell \in I_{h_2}} b_\ell^2 \right)^{1/2} \nonumber\\
& \ll a_{h_1} a_{h_2} \label{ah1ah2}
\end{align}
by Cauchy--Schwarz.\\

When $j_1$ and $j_2$ are fixed, there can only be solutions of $|j_1 z_m - j_2 z_n| <1$ with $z_m \in I_{h_1}$ and $z_n \in I_{h_2}$ for particular pairs $(h_1,h_2)$. Assume that $z_m \in I_{h_1}$ and $z_n \in I_{h_2}$ such that $|j_1 z_m - j_2 z_n| < 1$. Recall that $j_1 \geq j_2$ by assumption, so we have $\left| \frac{z_m}{z_n} - \frac{j_2}{j_1} \right| < \frac{1}{j_1 z_n}$ and consequently $\frac{z_m}{z_n} \leq \frac{j_2}{j_1} + \frac{1}{j_1 z_n} \leq 2$. Since $z_m \in I_{h_1}$ and $z_n \in I_{h_2}$, the quotient $z_m / z_n$ is somewhere between $(1+ 1/T)^{h_1-h_2-1}$ and $(1+ 1/T)^{h_1 -h_2+1}$, so that
\begin{equation} \label{quot_1}
\frac{z_m}{z_n} \leq  \underbrace{(1+ 1/T)^{h_1-h_2}}_{\leq 2 (1 + 1/T) \leq 3, \text{ since $z_m/z_n \leq 2$}} (1+ 1/T) \leq (1+ 1/T)^{h_1-h_2} + \frac{3}{T}.
\end{equation}
Similarly
\begin{equation} \label{quot_2}
\frac{z_m}{z_n} \geq (1+ 1/T)^{h_1-h_2} - \frac{3}{T}.
\end{equation} 

Since $j_1 \geq 2^{u-1}$ and $z_n \geq N^{1.01}$ by assumption, we have $\left| \frac{z_m}{z_n} - \frac{j_2}{j_1} \right| \leq \frac{1}{2^{u-1} N^{1.01}} \leq \frac{1}{T}$, where the last inequality follows from our choice of $T$. Overall, together with \eqref{quot_1} and \eqref{quot_2} this shows that the inequality $|j_1 z_m - j_2 z_n| <1$ for $z_m \in I_{h_1}$ and $z_n \in I_{h_2}$ is only possible when 
\[
\left| \left( 1 + \frac{1}{T} \right)^{h_1 - h_2} - \frac{j_2}{j_1} \right| \leq \frac{4}{T}.
\]
Note that, for fixed $j_1,j_2$, this is an inequality which only depends on $h_1,h_2$ and not on $z_m,z_n$ anymore. Thus in combination with \eqref{ah1ah2} we obtain
\begin{equation*}
\sum_{1 \leq m,n \leq M} \mathbf{1} \big(|j_1 z_m - j_2 z_n| < 1 \big) \ll \sum_{\substack{h_1, h_2 \geq 0, \\ \left|\left(1 + \frac{1}{T} \right)^{h_1-h_2} - \frac{j_2}{j_1} \right| \leq \frac{4}{T}}} a_{h_1} a_{h_2}
\end{equation*}
for all fixed $j_1$ and $j_2$. When summing over $j_1$ and $j_2$, we obtain the conclusion of Lemma \ref{lemma_new_2}.
\end{proof}

\begin{proof}[Proof of Lemma \ref{lemma_new_3}]
By the properties of $\widehat{K}$ and $\Phi$ we have
 \begin{align} \label{befint}
 & \sum_{2^{u-1} \leq j_1, j_2 < 2^u} \frac{1}{(j_1 j_2)^{1/2}} \sum_{\substack{h_1 \geq 0, h_2 \geq 0,\\\left| \left( 1 + \frac{1}{T} \right)^{h_1 - h_2} - \frac{j_2}{j_1} \right| \leq \frac{4}{T}}} a_{h_1} a_{h_2} \nonumber \\
& \ll \sum_{j_1,j_2 \geq 1} \frac{\widehat{K}(\log j_1j_2)}{(j_1 j_2)^{1/2}} \sum_{h_1,h_2 \geq 0} a_{h_1} a_{h_2} \widehat{\Phi}\left(T \log\left(\frac{j_1}{j_2}(1 + 1/T)^{h_1 - h_2}\right)\right) \nonumber \\
& \ll \frac{1}{T} \int_{\mathbb{R}} \sum_{j_1,j_2 \geq 1} \frac{\widehat{K}(\log j_1j_2)}{(j_1 j_2)^{1/2}} \bigg ( \frac{j_1}{j_2} \bigg )^{it} |P(t)|^2  \Phi(t/T) dt.
 \end{align} 
Note that we crucially used the fact that in all three lines of the displayed equation above, all terms in the summations are non-negative, because $\widehat{K}, \Phi$ and $\widehat{\Phi}$ are all non-negative. Thus, using that $2^u \ll (j_1 j_2)^{1/2} \ll 2^u$, we have
\[
\sum_{2^{u-1} \leq j_1, j_2 \leq 2^{u}} \sum_{\substack{h_1, h_2 \geq 0, \\ \left|\left(1 + \frac{1}{T} \right)^{h_1-h_2} - \frac{j_2}{j_1} \right| \leq \frac{4}{T}}} a_{h_1} a_{h_2} \ll \frac{2^u}{T} \int_{\mathbb{R}} \sum_{j_1,j_2 \geq 1} \frac{\widehat{K}(\log j_1j_2)}{(j_1 j_2)^{1/2}} \bigg ( \frac{j_1}{j_2} \bigg )^{it} |P(t)|^2  \Phi(t/T) dt,
\]
as claimed.
\end{proof}

In order to prove Lemma \ref{lemma_new_4}, we need the following technical tool, which is Lemma 5.3 of \cite{dlbT}.  \\
\begin{lemma}\label{fourier} Let $\sigma \in (-\infty,1)$ and let $F$ be a holomorphic function in the strip $y=\Im  z\in [\sigma-2,0]$, such that 
\begin{equation}\label{croissance} \sup_{\sigma-2 \leq y \leq 0}\lvert F(x+iy)\rvert \ll \frac{1}{x^2+1}.\end{equation}
Then for all $s=\sigma+it\in\mathbb{C}$, $t\neq0$, we have
\begin{align*}
& \sum_{k,\ell \geqslant 1}\frac{\widehat{F}(\log k \ell)}{k^{s} \ell^{\overline{s}}} \\
& = \int_{\mathbb{R}} \zeta(s+iu) \overline{\zeta(s-iu)} F(u)du +2\pi \zeta(1-2it) F(is-i)+2\pi \zeta(1+2it) F(i \overline{s} -i).
\end{align*}
\end{lemma}
The proof of this lemma (Lemma \ref{fourier}) is only briefly sketched in \cite{dlbT}. However, a detailed proof of a similar lemma is given in \cite[Lemma 1]{BSconv}; for the proof of our lemma one can exactly follow the argument given there, just using the function $f \colon z \mapsto \zeta(z+it) \zeta(z-it) K(i\sigma - iz)$ instead of the one considered there.\\

\begin{proof}[Proof of Lemma \ref{lemma_new_4}]
For simplicity of writing we define
\[G(t)= \sum_{j_1,j_2\geq 1}   \frac{\widehat{K}(\log j_1j_2)}{(j_1 j_2)^{1/2}} \left( \frac{j_1}{j_2} \right)^{it}.\]

We note that $K$ can be extended analytically and satisfies assumption \eqref{croissance}. Furthermore 
 \begin{equation}\label{Kbound} 
|K(t-i/2)| \ll N^{1+\ve/4}/(t+1)^2, \qquad |K(-t-i/2)|  \ll N^{1 + \ve/4}/(t+1)^2. 
\end{equation}

We first note that we have the pointwise bound
\begin{equation}\label{zerocontribution}
\vert P(t)\vert^2 \leq \vert P(0)\vert^2 =\left( \sum_{h \geq 0} a_h \right)^2 \ll \left( \sum_{n=1}^{\infty} b_n \right)^2 \ll N^4, \qquad t \in \mathbb{R}.
\end{equation}
This allows us to see that 
\begin{equation*}
    \int_{-1}^1 G(t) |P(t)|^2 \Phi(t/T) dt \ll N^5
\end{equation*}
and thus we can restrict our integration domain to $|t|\ge 1$ below. By Lemma \ref{fourier} we have 
\[
\int_{|t|\ge 1} G(t) \vert P(t)\vert^2 \Phi(t/T) dt = \mathrm{Int}_1 + \mathrm{Int}_2 + \mathrm{Int}_3,
\]
where 
\begin{align*}
\mathrm{Int}_1 &=  \int_{|t|\ge 1} \vert P(t)\vert^2 \Phi(t/T) \int_{\mathbb{R}}\zeta(1/2+it+iu) \zeta(1/2-it+iu)K(u) ~du ~dt,\\
\mathrm{Int}_2 & =  2\pi\int_{|t|\ge 1} \zeta(1-2it)K(-t-i/2)\vert P(t)\vert^2 \Phi(t/T)~dt,  \\
\mathrm{Int}_3 & =  2\pi\int_{|t|\ge 1} \zeta(1+2it)K(t-i/2)\vert P(t)\vert^2 \Phi(t/T)~dt.
\end{align*}
 
Using \eqref{Kbound} 
together with the easy estimate $|\zeta(1 + it)| \ll \log t$, we obtain 
\begin{align*} 
\mathrm{Int}_2 & \ll N^{1 + \ve/4} N^4 \int_{t \geq 1} \frac{\log t}{t^2} ~dt  \\ 
& \ll N^{5 + \ve/4}.
\end{align*}
Exactly the same estimate holds for $\mathrm{Int}_3$.\\
  
The classical convexity bound $\vert \zeta(1/2+it)\vert \ll \vert t\vert^{1/4}$ gives $|\zeta(1/2+it+iu) \zeta(1/2-it+iu)\vert \ll (\vert t\vert + \vert u\vert)^{1/2} \ll |t|^{1/2} + |u|^{1/2}$. Hence we can bound the contribution of the domain $\vert u\vert \geq T$ to $\mathrm{Int}_1$ by
\begin{align*}
\int_{\mathbb{R}} |t|^{1/2} \vert P(t)\vert^2 \Phi(t/T) \underbrace{\int_{\vert u\vert \geq T} K(u) ~du}_{\ll T^{-1}} ~dt + \int_{\mathbb{R}} \vert P(t)\vert^2 \Phi(t/T) \underbrace{\int_{\vert u\vert \geq T} \vert u\vert^{1/2}K(u) ~du}_{\ll T^{-1/2}} ~dt \\
\ll  \int_{\mathbb{R}} \left(\frac{\vert t\vert^{1/2}}{T} + \frac{1}{T^{1/2}} \right) \vert P(t)\vert^2 \Phi(t/T)dt  \ll T^{1/2} E_N^*, \nonumber
\end{align*}
where we used $K(u) \ll u^{-2}$ and Lemma \ref{lemma_new_1} together with the quick decay of $\Phi$. 

Let $A=\frac{178}{13}$. Then by Ivi\'c's theorem \cite[Theorem 8.3]{Ivic} we have
\begin{equation} \label{zeta_bound}
\int_0^T |\zeta(1/2 + it)|^A dt \ll T^{2 + \frac{3(A-12)}{22} +\ve} = T^{29/13 + \ve}.
\end{equation}
The contribution to $\mathrm{Int}_1$ of small $u$ is
 \begin{align*}
\int_{\vert u\vert \leq T} K(u) \left(\int_{\mathbb{R}} \vert \zeta(1/2+it+iu)\vert \vert\zeta(1/2-it+iu)\vert \vert P(t)\vert^2 \Phi(t/T)   ~dt\right) ~du.   
\end{align*} 

To estimate the term inside the brackets, we  use H\"older's inequality with parameters $1/A+1/A+1/B=1$, so that $B = \frac A{A-2} = \frac{89}{76}$, and write $\vert P(t)\vert^{2} = \vert P(t)\vert^{2-2/B} \vert P(t)\vert^{2/B}$. By Lemma \ref{lemma_new_1} together with \eqref{zerocontribution} and \eqref{zeta_bound} we deduce 
\begin{align}
& \int_{\mathbb{R}} \vert \zeta(1/2+it+iu)\vert \vert\zeta(1/2-it+iu)\vert \vert P(t)\vert^2 \Phi(t/T) ~dt \label{note_here}\\
& \ll \left( \int_{\mathbb{R}}  \vert \zeta(1/2+it+iu)\vert^{A} \Phi(t/T) ~dt \right)^{1/A}\left( \int_{\mathbb{R}}  \vert \zeta(1/2-it+iu)\vert^{A} \Phi(t/T) ~dt\right)^{1/A} \times \nonumber\\ 
& \times  \vert P(0)\vert^{2(1-1/B)} \left(\int_{\mathbb{R}} \vert P(t)\vert^2 \Phi(t/T)~dt\right)^{1/B}   \nonumber\\
& \ll \left(T^{29/13 + \ve} \right)^{2/A} N^{4(1-1/B)} T^{1/B} (E_N^*)^{1/B} \nonumber\\
&= \left(T^{29/13 + \ve} \right)^{2/A} N^{8/A} T^{(A-2)/A} \left( E_N^* \right)^{(A-2)/A}. \nonumber
\end{align}
Integrating over $u$ we deduce that 
\begin{align*}
\mathrm{Int}_1 &\ll \left(T^{29/13 + \ve} \right)^{2/A} N^{8/A} T^{(A-2)/A} (E_N^*)^{(A-2)/A}  \\
&= T^{105/89 + \ve} N^{52/89} (E_N^*)^{76/89}.
\end{align*}
Using Lemma \ref{lemma_new_2} and \eqref{befint} and inserting our bounds for $\mathrm{Int}_1, \mathrm{Int}_2$ and $\mathrm{Int}_3$  we obtain
\begin{align*}
& \sum_{1 \leq m,n \leq M}\sum_{2^{u-1} \leq j_1, j_2 < 2^u}  \mathbf{1} \big(|j_1 z_m - j_2 z_n| < 1 \big) \\\
& \ll \frac{2^u}{T} \int_{\mathbb{R}} G(t) \vert P(t)\vert^2 \Phi(t/T) dt \\
& \ll \frac{2^u}{T}\left( N^{5 + \ve/4} + T^{105/89 + \ve} N^{52/89} (E_N^*)^{76/89} + T^{1/2} E_N^* \right) \\
& \ll \frac{2^u}{T}\left( N^{5 + \ve/4} + T^{105/89 + \ve} N^{52/89} (E_N^*)^{76/89} \right),
\end{align*} 
where we used that the term $T^{1/2} E_N^*$ is dominated by the other two summands. Substituting $T = 2^u N^{1+\ve/2}$ and using that $2^u \ll N$, we finally get the upper bound
\begin{align}\label{finalcase1}
&\sum_{1 \leq m,n \leq M}\sum_{2^{u-1} \leq j_1, j_2 < 2^u}  \mathbf{1}(|j_1 z_m - j_2 z_n| < 1) \nonumber  \\
&\ll \frac 1{N^{1+\ve/2}} (N^{5+\ve/4} + N^{52/89} N^{2 \times 105/89 + 2\ve} (E_N^*)^{76/89}) \nonumber \\
&= N^{4-\ve/4} + N^{173/89+ 2\ve} (E_N^*)^{76/89}. \nonumber
\end{align}
This establishes Lemma \ref{lemma_new_4}.
\end{proof}

We note here that conditionally under the Lindel\"of hypothesis we could estimate the integral in line \eqref{note_here} much more efficiently, by using a pointwise bound for the zeta function and estimating the remaining integral with Lemma \ref{lemma_new_1}.\\

Combining Lemmas \ref{lemma_new_2}, \ref{lemma_new_3} and \ref{lemma_new_4}, we have shown in our Case 1 analysis that the contribution of pairs $z_m,z_n$ with $z_m,z_n \geq N^{1.01}$ to the counting problem \eqref{to_sum} is 
\begin{equation} \label{to_sum_new}
\sum_{\substack{1 \leq m,n \leq M,\\z_m,z_n  \geq N^{1.01}}} ~\sum_{\substack{2^{u-1} \leq j_1,j_2 < 2^u}} ~\mathbf{1} \big(|j_1 z_m - j_2 z_n| < 1 \big) \ll N^{4-\ve/4} + N^{173/89+ 2\ve} (E_N^*)^{76/89}.
\end{equation}

Before we move on to Case 2, we make some further comments on our argument above. Intuitively, it would seem more natural to work with $Q(t):=\sum_m z_m^{it}$ rather than with the more complicated function $P(t)$. However, from a technical point of view the key problem in the whole argument is to be able to choose an appropriate value of $T$ which balances the contribution to our integrals of those values of $t$ for which $|t|$ is ``large'' (this gets worse when $T$ is larger, since the bound for the zeta function grows polynomially in $T$) against the contribution coming from those $t$ for which $|t|$ is ``small'' (this contribution can only be compensated in the final estimate when $T$ is sufficiently large). When working directly with $Q$, the size of $T$ would need to depend on the size of the $z_m$ in order to be able to control $\int |Q|^2$. The ``orthogonalization'' procedure leading to our definition of $P(t)$ gives us more freedom in our choice of $T$. The whole problem discussed in this paragraph occurs only in the real-number setting, in contrast to the integer setting.\\

$\bullet$ {\bf Case 2:~Counting solutions for $z_m,z_n$ with $\min \{z_m,z_n\} < N^{1.01}$.}\\
 
First we consider the contribution to \eqref{to_sum} of those $z_m$ and $z_n$ for which $\max\{z_m,z_n\} < 4N^{1/4}$. We assumed that $x_{n+1} - x_n \geq c > 0$, so $z_n \geq c$ for all $n$. Furthermore, we deduce that among $z_1, \dots, z_M$ there are at most $\ll N^{5/4}$ many elements which are smaller than $4 N^{1/4}$ (we suppress the dependence of the implied constant on $c$). Note that whenever $j_1$ and $z_m,z_n$ are fixed, there are at most $\ll 1$ many possible choices for $j_2$ such that $|j_1 z_m - j_2 z_n| < 1$, again since $z_n \geq c$. Thus the total contribution of pairs $z_m,z_n$ with $\max\{z_m,z_n\} < 4N^{1/4}$ to our counting problem is at most
\begin{equation*}
\sum_{\substack{1 \leq m,n \leq M \\ \max\{z_m,z_n\} < 4 N^{1/4}}} \sum_{\substack{2^{u-1} \leq j_1,j_2 < 2^u}} \mathbf{1} \big(|j_1 z_m - j_2 z_n| < 1 \big) \ll \left( N^{5/4} \right)^2 2^u \ll  N^{7/2}.
\end{equation*}

Now consider the case when $\max \{z_m,z_n\} \geq 4 N^{1/4}$. Recall that we have localized $j_1,j_2$ into a dyadic interval in the counting problem. This implies a similar localization for $z_m$ and $z_n$, since $j_1/j_2 \in [1/2,2]$ and $|j_1 z_m - j_2 z_n|<1$ are only possible if we have $z_m / z_n \in [1/4,4]$, given the fact that $\max\{ z_m,z_n \} \geq 4 N^{1/4}$. \\

Thus we can restrict ourselves in the counting problem \eqref{to_sum} to the case when $z_m \in [4 N^\beta, 8N^\beta)$ for some $\beta \geq 1/4$, and when consequently $z_n$ needs to be in $[N^\beta, 32 N^\beta)$. Note that there are $\ll \log N$ many intervals of this form necessary to cover the whole relevant range $[N^{1/4}, N^{1.01}]$, and clearly we only need to consider $1/4 \leq \beta \leq 1.01$. We count more solutions if we relax the condition to $z_m,z_n \in [N^\beta, 32 N^\beta)$. Thus, let us consider 
\begin{equation} \label{bound_this}
\sum_{\substack{1 \leq m,n \leq M,\\ z_m,z_n \in [N^\beta, 32 N^\beta)}} \sum_{\substack{2^{u-1} \leq j_1,j_2 < 2^u}} \mathbf{1} \big(|j_1 z_m - j_2 z_n| < 1 \big)
\end{equation}
for some $\beta \in [1/4,1.01]$. We set up everything as in Case 1, but now we define $T = 2^u N^{\beta}$. We define the $b_k$'s as before but restricting ourselves to those $z_m$ contained in $[N^\beta, 32 N^\beta)$. That is, we set
$$
b_k = \sum_{\substack{1 \leq m \leq M,\\z_m \in [N^\beta, 32 N^\beta)}} \mathbf{1}(z_m \in [k,k+1)).
$$
Note that previously we had $\sum_k b_k = M \leq  N^2$, whereas now we have a stronger bound. Applying the Cauchy--Schwarz inequality we obtain
\begin{equation} \label{smaller_at_zero}
\sum_{k} b_k  \ll \sqrt{E_N^*} N^{\beta/2}.
\end{equation}
We define $(a_h)_{h \geq 0}$ and $P(t)$ as in Case 1; see \eqref{ah} and \eqref{P_def}. In the present case the inequality $|j_1 z_m - j_2 z_n|<1$ is only possible when $\left| \frac{z_m}{z_n} - \frac{j_2}{j_1} \right| \leq \frac{1}{2^{u-1} N^\beta}$. By construction, $2^{u-1} N^\beta$ becomes large in comparison with $T$, and we can continue to argue as in Case 1. Note that now, as a consequence of \eqref{smaller_at_zero},  we have $|P(0)|^2 \ll E_N^* N^\beta$ instead of $|P(0)|^2 \ll N^4$ as in Case 1. Proceeding as in Case 1 we obtain 
\[
\mathrm{Int}_2, \mathrm{Int}_3\ll N^{1+\beta+\ve/4}E_N^* + T^{1/2} E_N^*.
\]
As for $\mathrm{Int}_1$, we now obtain, again writing $A = \frac{178}{13}$, 
\begin{align}
\mathrm{Int}_1 &\ll (T^{29/13 + \ve})^{2/A} (E_N^* N^\beta)^{2/A} (E_N^*)^{1-2/A} T^{1-2/A} \nonumber\\
&= E_N^* (T^{29/13 + \ve})^{2/A} (N^\beta)^{2/A} T^{1-2/A}. \nonumber
\end{align}
Substituting $T = 2^u N^{\beta} \ll N^{1+\beta}$, we conclude that \eqref{bound_this} is bounded by
\begin{align}
& \ll \frac{2^u}{T} \left(E_N^* N^{1+\beta+\ve/4} + E_N^*  T^{105/89+\ve} N^{13 \beta /89} \right) \nonumber\\
& \ll E_N^* N^{1+\ve/4}  + E_N^* N^{105/89 + 29 \beta / 89 + (1+\beta)\ve}. \nonumber
\end{align}
Recall that we only need to consider $\beta \leq 1.01$, and that there are at most $\ll \log N$ many different values of $\beta$ to consider for Case 2. Hence it follows that, for every fixed value of $u$, \eqref{bound_this} is bounded by
\begin{equation}\label{finalsec2} \ll  E_N^* N^{1.51+3\ve}.   \end{equation} We could have used a different value of $A$ here (for instance $A = 12$, namely Heath-Brown's bound on the twelfth moment from \cite{HB}) to  arrive at \eqref{finalsec2}. However, we kept the same parameters as in Case $1$ to simplify the writing. \\

Finally, inserting in \eqref{line_1} resp.\ \eqref{to_sum} the bound \eqref{to_sum_new} from Case 1 together with the bound \eqref{finalsec2} yields
\[
\mathrm{Var}(h_N, \mu) \ll \max \left(N^{-\ve/8} + N^{-183/89 + 3\ve} (E_N^*)^{76/89}, E_N^* N^{-2.49 + 4\ve} \right),
\]
which concludes the proof of Lemma \ref{lem:var_bound}.

\section{Proof of Theorem \ref{th1}: conclusion of the proof} \label{sec_end}
The crucial ingredient in the proof of Theorem \ref{th1} is the variance bound from Lemma \ref{lem:var_bound}. We record that we have, for every sufficiently small $\ve > 0$,
\[
\mathrm{Var}(h_N, \mu) \ll \max \left(N^{-\ve/8} + N^{-183/89 + 3\ve} (E_N^*)^{76/89}, E_N^* N^{-2.49 + 4\ve} \right).
\]

Inserting $E_N^* \ll N^{183/76-\delta}$  shows that for every sufficiently small $\delta' > 0$ we have
\begin{equation} \label{var_bound}
\mathrm{Var}(h_N, \mu)  \ll N^{-\delta'}.
\end{equation}
 Everything else now follows from a standard procedure. To be a bit more specific, convergence in \eqref{conv_f} can be established using the estimate for the expectations in Section \ref{sec_av}, and using the variance bound  \eqref{var_bound} together with Chebyshev's inequality and the Borel--Cantelli lemma. From that we get a convergence result for almost all $\alpha \in \mathbb{R}$, for fixed values of $r$ and $s$. One notes that there are only countably many possible values of $r$, and that by continuity/monotonicity it is sufficient to consider countably many values of $s$. Since a countable union of sets of measure zero has measure zero as well, almost all $\alpha \in \mathbb{R}$ have the property that \eqref{conv_f} holds for all $r$ and all $s$, as desired. We refer the reader to \cite{all} or \cite{rt}, where this argument is carried out in full detail. It applies without any modifications to the situation in the present paper.

\section{Proof of Theorem \ref{th2_new}.} \label{sec_new}

We set up the same machinery as in the proof of Theorem \ref{th1}. Controlling the expectations, as in Section \ref{sec_av} above, is unproblematic. The crucial part is again the variance estimate. As in Section \ref{sec_var}, we are led to the counting problem 
\begin{equation*} \label{to_sum_theta}
\sum_{1 \leq m,n \leq M} ~\sum_{\substack{2^{u-1} \leq j_1,j_2 < 2^u}} ~\mathbf{1} \big(|j_1 z_m - j_2 z_n| < 1 \big)
\end{equation*}
where $\{z_1, \dots, z_M\}$ is the multi-set of all the absolute differences $\{|x_m - x_n|:~1 \leq m, n \leq N,~m \neq n\}$. As above, $u$ is a positive integer with $2^{u} \leq 2 r N$, and $M = N^2 - N$.\\

As in the general argument before, we can easily dispose of the contribution of those $z_m,z_n$ for which $\max\{z_m,z_n\} < 4 N^{1/4}$. Thus again we can localize $z_m$ and $z_n$, and restrict ourselves to counting 
\begin{equation} \label{to_sum_theta_j}
\sum_{\substack{1 \leq m,n \leq M,\\z_m,z_n \in [N^\beta, 32 N^\beta)}} ~\sum_{\substack{2^{u-1} \leq j_1,j_2 < 2^u}} ~\mathbf{1} \big(|j_1 z_m - j_2 z_n| < 1 \big)
\end{equation}
for some $1/4 \leq \beta \leq 1.01$, where we understand that the final range for $\beta=1.01$ extends over all of $[N^\beta,\infty)$ rather than only $[N^\beta,32 N^\beta)$. Note that in total at most $\ll \log N$ many different values of $\beta$ need to be considered. Let $\beta \in [1/4,1.01]$ be fixed. Let $u$ in \eqref{to_sum_theta_j} be fixed. For integers $k \geq 0$ we define 
\begin{equation} \label{b_def_ps}
b_k = \sum_{\substack{1 \leq m \leq M,\\z_m \in [N^\beta, 32 N^\beta)}} \mathbf{1} \left(z_m \in \left[\frac{k}{2^u},\frac{k+1}{2^u} \right) \right),
\end{equation}
where again in case $\beta=1.01$ the range $[N^\beta, 32 N^\beta)$ is understood to be replaced by $[N^\beta,\infty)$. Note the difference in comparison with \eqref{b_k_def}. There we collected all $z_m$ in a range of the form $[k,k+1)$, since we could only control the number of solutions of the specific inequality \eqref{dioph_in_1}, which has ``$<1$'' on the right-hand side. In contrast we can now control the number of solutions on a finer scale, and can accordingly set shorter ranges for the grouping of the $z_m$ (where $\gamma = 2^{-u}$).\\

Let $\ve>0$ be a small constant (chosen depending on the size of $\eta$ and $\delta$ in the statement of the theorem). Set $T = 2^u N^{\min\{\beta - \ve, 1 + \ve\}}$. Unlike the argument in the general case in Section \ref{sec_var}, we do not explicitly distinguish between Case 1 and Case 2, but have implicitly included this distinction into the way that $T$ is defined. As in Section \ref{sec_var}, we split the interval $[1, \infty)$ into a disjoint union $\bigcup_{h=0}^\infty I_h$, where
$$
I_h = \left[ \left\lceil \left(1 + \frac{1}{T} \right)^h \right\rceil , \left\lceil \left(1 + \frac{1}{T} \right)^{h+1} \right\rceil \right),
$$
and set 
\begin{equation*}
a_h = \left( \sum_{\substack{k:~k/2^u \in I_h}} b_k^2 \right)^{1/2}, \qquad h \geq 0,
\end{equation*}
as well as
\begin{equation*}
P(t) = \sum_{h=0}^\infty a_h \left(1 + \frac{1}{T} \right)^{iht}.
\end{equation*}
Then by construction we again have
\begin{equation*}\label{orth_new}
\int_\R |P(t)|^2 \Phi(t/T)dt \ll T \sum_{h=0}^\infty a_h^2
\end{equation*}
as during the proof of Lemma \ref{lemma_new_1} in Section \ref{sec_var}, but now we can continue to estimate this by \eqref{engamma_ass} and obtain
\begin{equation} \label{note_that}
T \sum_{h \geq 0} a_h^2 \ll T \sum_{k \geq 0} b_k^2 \ll T E_{N,2^{-u}}^* \ll T \left( N^{2 + \eta} + 2^{-u} N^{3 - \delta} \right).
\end{equation}

Now we establish the necessary upper bound on $|P(0)|$. Trivially we always have $|P(0)| \ll N^2$. For small values of $\beta$ we obtain a better estimate. Note that 
\begin{eqnarray}
|P(0)| & = & \sum_h a_h \leq \sum_k b_k = \# \left\{m:~z_m \in [N^\beta,32 N^\beta) \right\} \nonumber\\
& \leq & \sum_{a=\lfloor N^\beta \rfloor}^{\lfloor 32 N^\beta \rfloor} \# \left\{m:~z_m \in [a,a+1) \right\} \nonumber\\
& \ll & N^{\beta/2} \left( \sum_{a=\lfloor N^\beta \rfloor}^{\lfloor 32 N^\beta \rfloor} \Big(\# \left\{m:~z_m \in [a,a+1) \right\} \Big)^2 \right)^{1/2} \label{CS} \\
& \ll & N^{\beta/2} \sqrt{E_{N,1}^*} \nonumber\\
& \ll & N^{\beta/2} N^{(3-\delta)/2} = N^{3/2 + \beta/2 - \delta/2}, \label{CS2}
\end{eqnarray}
where for \eqref{CS} we used Cauchy--Schwarz and for \eqref{CS2} our assumption \eqref{engamma_ass} in the statement of Theorem \ref{th2_new}. So overall we have $|P(0)| \ll N^{\min\{2,3/2 + \beta/2 - \delta/2\}}$.\\

As in the proof of Lemma \ref{lemma_new_2} in Section \ref{sec_var}, we assume without loss of generality that $j_1 \geq j_2$. Recall that $j_1,j_2 \geq 2^{u-1}$ by assumption. Assume that $z_m \in [k/2^u,(k+1)/2^u)$ for some $k$. Then the inequality $|j_1 z_m - j_2 z_n | \leq 1$ is only possible when 
$$
\left| \frac{\left\lceil \frac{j_1 k}{j_2} \right\rceil}{2^u} - z_n \right| \leq \frac{5}{2^{u}},
$$
which is a version of \eqref{cons} that is adapted to the construction in \eqref{b_def_ps}. Arguing as in the lines leading to \eqref{ah1ah2}, this again gives
\begin{equation*}
\sum_{z_m \in I_{h_1}, z_n \in I_{h_2}} \mathbf{1} \big(|j_1 z_m - j_2 z_n| < 1 \big) \ll a_{h_1} a_{h_2},
\end{equation*}
which perfectly resembles \eqref{ah1ah2} but where now the $b_k$ and $a_h$ are defined in a different way according to \eqref{b_def_ps}. Note that $T$ is chosen in such a way that $j_1 z_m$ and $j_2 z_n$ exceed $T$; indeed, by assumption $j_1,j_2 \geq 2^{u-1}$ and $z_m,z_n \geq N^\beta$, while $T \leq 2^u N^{\beta - \ve}$ by definition. Thus we can continue the argument as in Section \ref{sec_zeta}. It turns out that in this setting in order to bound $\text{Int}_1$ it is sufficient to use $\vert \zeta(1/2+it) \vert \ll \vert t\vert^{1/6}$(which is essentially the Weyl--Hardy--Littlewood bound), rather than the more elaborate argument relying on estimates for moments of the Riemann zeta function. We obtain
\begin{align*}
& \sum_{\substack{1 \leq m,n \leq M,\\z_m,z_n \in [N^\beta, 32 N^\beta)}} ~\sum_{\substack{2^{u-1} \leq j_1,j_2 < 2^u}} ~\mathbf{1} \big(|j_1 z_m - j_2 z_n| < 1 \big) \\
& \ll \frac{2^u}{T} \left( |P(0)|^2 N^{1 + \ve/4} +  T^{1+1/3} E_{N,2^{-u}}^* \right) \\
& \ll \frac{2^u N^{\min\{4,3+\beta-\delta\}} N^{1+\ve/4}}{T} + (2^u)^{4/3} N^{(1 + \ve)/3}  N^{2+\eta} + (2^u)^{4/3} N^{(1 + \ve)/3} 2^{-u} N^{3-\delta} \\
& \ll N^{1 + \ve/4 + \min\{4,3+\beta-\delta\} - \min\{\beta-\ve,1 + \ve\}}+ N^{11/3 + \ve + \eta}  + N^{11/3 + \ve - \delta} \\
& \ll N^{4-\ve/2}
\end{align*}
if $\ve$ was chosen sufficiently small (with respect to $\delta$). Here we used \eqref{engamma_ass} as well as $T \ll 2^u N^{1 + \ve},~2^u \ll N$. Noting that we need to consider at most $\ll \log N$ different values of $\beta$, this gives the necessary variance estimate. The remaining part of the proof of Theorem \ref{th2} can be carried out exactly as in the proof of Theorem \ref{th1}. We remark that any subconvex bound for the Riemann zeta function would be sufficient to derive the same conclusion.

\section{Proof of Theorem \ref{th2}.}  \label{sec_theta}

We assume that $\theta > 1$ is fixed, and consider the sequence $(x_n)$ defined by $x_n = n^\theta,~n \geq 1$. Note that with this definition we have $x_{n+1} - x_n \geq 1$ for all $n \geq 1$, so the assumption $x_{n+1} - x_n \geq c$ of Theorem \ref{th2_new} is satisfied in this case with $c=1$. The following lemma of Robert and Sargos shows that the necessary bound on $E_{N,\gamma}^*$ also is satisfied for this sequence.\footnote{We thank Niclas Technau for pointing out to us that the estimate in Lemma \ref{rs_lemma} is also contained as a special case in a general result in a very recent paper of Huang \cite{huang}. Huang's result gives improved error terms, but for our application this does not play a role. However, the generality of Huang's results could allow further applications of our method in the spirit of our Theorem \ref{th2}.}

\begin{lemma}[{\cite[Theorem 2]{rs}}] \label{rs_lemma}
Let $\theta \neq 0,1$ be a fixed real number. For any $\gamma > 0$ and $B \geq 2$, let $\mathcal{N}(B,\gamma)$ denote the number of 4-tuples $(n_1, n_2, n_3, n_4) \in \{B+1, B+2, \dots, 2B\}^4$ for which
\begin{equation} \label{gamma_equ}
\left| n_1^\theta - n_2^\theta + n_3^\theta - n_4^\theta \right| \leq \gamma.
\end{equation}
Then for every $\varepsilon > 0$, 
$$
\mathcal{N} (B,\gamma) \ll_\ve B^{2 +\ve} + \gamma B^{4 - \theta + \ve}. 
$$
\end{lemma}

The restriction to a dyadic range for $(n_1, n_2, n_3, n_4)$ in the statement of the lemma does not actually play a role. This is easily seen by interpreting the number of solutions of the inequality as an $L^4$-norm. Indeed, generalizing the definition in \eqref{measure} and setting 
$$
d\mu_{2 \gamma}(x) = \frac{(\sin (\gamma x))^2}{\pi \gamma x^2} dx,
$$
we have a measure whose Fourier transform is a (normalized) tent function on $[-2\gamma,2 \gamma]$. Let $E_{N,\gamma}^*$ denote the number of solutions of \eqref{gamma_equ}, subject to $(n_1, n_2, n_3, n_4) \in \{1, \dots, N\}^4$. Assume for simplicity of writing that $N$ is a power of $2$, i.e.\ $N = 2^L$ for some $L \geq 1$. Then applying H\"{o}lder's inequality we obtain
\begin{align}
E_{N,\gamma}^* & \ll \int_\R \left(\sum_{\ell=0}^{L} ~\sum_{2^{\ell-1} < n \leq 2^{\ell}} ~e^{2 \pi i n^\theta x} \right)^4 d\mu_{2 \gamma}(x) \nonumber\\
& \ll \int_\R \left(\sum_{\ell=0}^{L} 1\right)^3 ~\sum_{\ell=0}^{L} \left\vert \sum_{2^{\ell - 1} < n \leq 2^{\ell}} e^{2 \pi i n^\theta x} \right\vert^4 d\mu_{2 \gamma}(x) \nonumber \\
& \ll (\log N)^3 \sum_{\ell=0}^{L}  \int_\R \left\vert \sum_{2^{\ell - 1} < n \leq 2^{\ell}}e^{2 \pi i n^\theta x} \right\vert^4 d\mu_{2 \gamma}(x) \label{this_line}\\
&\ll (\log N)^3  \sum_{\ell=0}^{L} E_{2^{\ell-1},2\gamma}  \nonumber \\
& \ll_\ve  N^{2 +\ve} + \gamma N^{4 - \theta + \ve} , \nonumber
\end{align}
which is obtained by interpreting the integrals in line \eqref{this_line} in terms of solutions of the Diophantine inequality \eqref{gamma_equ}, and applying Lemma \ref{rs_lemma} with para\-meters $2 \gamma$ and $B = 2^{\ell-1}$. Thus we have
\begin{equation} \label{engamma}
E_{N,\gamma}^* \ll_\ve N^{2 +\ve} + \gamma N^{4 - \theta + \ve}
\end{equation}
for any $\ve>0$. Consequently all assumptions of Theorem \ref{th2_new} are satisfied, and we can conclude that $(n^\theta \alpha)_n$ has Poissonian pair correlation for almost all $\alpha$.

\section{Closing remarks} \label{sec_close}

As remarked in the introduction, our method breaks down completely when the growth order of the sequence $(x_n)_{n \geq 1}$ is only linear or even slower. Not only does the ``lattice point counting with the zeta function'' argument from Section \ref{sec_zeta} fail to work in this situation, but there is a much more fundamental reason why the whole approach based on calculating first and second moments (expectations and variances, as in Sections \ref{sec_av} and \ref{sec_var}) fails to work in this setup. To give a brief sketch of what causes the problem, assume that $(x_n)_{n}$ is a sequence of reals such that $x_n \leq n,~ n \geq 1$. Assume that we want to bound the variance in analogy with \eqref{wish_to}, so say we want to show that
\begin{equation} \label{wish_to_2}
\int_\R \left(\frac{1}{N} \sum_{\substack{1 \leq m, n \leq N, \\ m \neq n}} \mathbf{1}_{[-1/N,1/N]} (x_m \alpha - x_n \alpha)\right)^2 \dmu
\end{equation}
tends to zero as $N \to \infty$ (where we write the original indicator function instead of its approximation by a trigonometric polynomial, and where for simplicity of writing we set $s=1$). By our assumption on the growth of $(x_n)_n$, all differences $x_m -x_n$ appearing in the sum above are uniformly bounded by $N$. Thus we have  $\mathbf{1}_{[-1/N,1/N]} (x_m \alpha - x_n \alpha) = 1$ throughout the range $\alpha \in [-1/N^2,1/N^2]$, for all $m,n \leq N$. Consequently
\begin{align*}
\int_\R \left(\frac{1}{N} \sum_{\substack{1 \leq m, n \leq N, \\ m \neq n}} \mathbf{1}_{[-1/N,1/N]} (x_m \alpha - x_n \alpha)\right)^2 \dmu & \geq \int_{-1/N^2}^{1/N^2} \left(\frac{1}{N} \sum_{\substack{1 \leq m, n \leq N, \\ m \neq n}}  1 \right)^2 \dmu \\
& \gg 1.
\end{align*}
Thus the variance fails to tend to zero for a slowly growing $(x_n)_n$, due to the fact that the contribution of small values of $\alpha$ to the variance integral is too large.\footnote{A similar argument appears at the end of \cite{rsarnak}, where it is used to show that the $L^2$ approach fails to work in the case of the triple correlation of $(n^2 \alpha)_n$; cf.\ also \cite{tw}.} The argument used in Section \ref{sec_zeta} fails to work in a similar way for slowly growing $(x_n)_n$, since the error terms coming from the contribution to the integrals of values of $t$ near zero become too large.\footnote{It might be difficult to spot at a quick glance, so we briefly comment on where the speed of growth of $(x_n)_n$ was used in our argument in Sections \ref{sec_zeta} and \ref{sec_new}. There is a term $|P(0)|^2 N^{1 + \ve/4}$ coming from the contribution of values of $t$ near the origin to the integral. This term is divided by $T$ at the end of the calculation, so we cannot take $T$ too small since we need $N |P(0)|^2 N^{1 + \ve/4} / N^4 T \to 0$. On the other hand, we cannot take $T$ too large, since we need $T \ll 2^u z_n$ to be able to detect the solutions of our Diophantine inequality. To balance everything out, we need to be able to assure that there are not too many small values of $z_n$ (i.e., not too many differences $x_m - x_n$ which are ``small''). In our proof of Theorem \ref{th1} our assumption on the order of the additive energy takes care of this: it is easy to see that an upper bound on $E_N^*$ implies an upper bound on the number of ``small'' differences $x_m - x_n$, which is what we used in Case 2 of Section \ref{sec_zeta}. In the setting of Theorem \ref{th2_new} a similar argument based on the energy assumption allowed us to control the number of small differences $x_m - x_n$; the relevant equations there are \eqref{smaller_at_zero} and \eqref{CS2}.} Consequently, it seems that for establishing Poissonian pair correlation of $(x_n \alpha)_n$ for almost all $\alpha$ for slowly growing $(x_n)_n$ some genuine new ideas are necessary. Note that we cannot simply remove all values of $\alpha$ near zero from the variance integral \eqref{wish_to_2} by replacing $\mu$ with some other measure which vanishes for small $\alpha$, since such a measure would fail to have non-negative Fourier transform (thereby causing major problems in other places). Note also that all these problems with slowly growing sequences $(x_n)_n$ are a novel aspect which only shows up in the real-number setup -- in contrast, when $(a_n)_n$ is an integer sequence which grows at most linearly, then $(a_n \alpha)_n$ is known to fail to have Poissonian pair correlation for any $\alpha$, because the additive energy of $(a_1, \dots, a_N)$ necessarily is of maximal possible order (cf.\ \cite{ls}).\\
 
We emphasize that the fact that our method fails to work in the case of slowly growing sequences $(x_n)_n$ should \emph{not} be understood as indicating that in such a case $(x_n \alpha)_n$ should necessarily fail to have Poissonian pair correlation for almost all $\alpha$. Quite on the contrary, there are good reasons to expect that also for slowly growing $(x_n \alpha)_n$ one should in ``generic'' situations obtain Poissonian pair correlation for almost all $\alpha$. It seems that the property of having Poissonian pair correlation for $(x_n \alpha)_n$ for almost all $\alpha$ can only be prevented by a certain (``small-scale'') combinatorial obstruction, in such a way that the case of integer sequences $(x_n)_n$ with slowly growing $(x_n)_n$ can be seen as a degenerate situation exhibiting exactly this type of combinatorial obstruction (coming from the fact that in the integer setup everything which is smaller than one in absolute value necessarily equals zero). We believe that these are very interesting phenomena, and we propose the following open problems.\\

{\bf Open Problem 1:} Let $\theta \in (0,1)$. Show that $(n^\theta \alpha)_{n \geq 1}$ has Poissonian pair correlation for almost all $\alpha$. Note that Lemma \ref{rs_lemma} is still valid for this range of $\theta$. \\

{\bf Open Problem 2:} Let $x_n = n + \log n$. Show that $(x_n \alpha)_n$ has Poissonian pair correlation for almost all $\alpha$. We note that it is possible to establish a variant of Lemma \ref{rs_lemma} for this setting (with exponent $3$ in place of $4-\theta$).\\

{\bf Open Problem 3:} Let $x_n = n \log n,~n \geq 1$. Show that $(x_n \alpha)_n$ has Poissonian pair correlation for almost all $\alpha$.\\

Clearly the exponent $183/76 - \delta$ in the statement of Theorem \ref{th1} is not optimal, and most likely it can be improved to $3 - \delta$ (which is the case conditionally under the Lindel\"of hypothesis). It seems to us that the method of Bloom and Walker \cite{bw}, which led to a quantitative improvement of the results of \cite{all}, cannot be used here. Their method relied on sum-product estimates, which, roughly speaking, leads to an integrand $|P(t)|^2$ being replaced by $|P(t)|^4$. In the case of integer sequences (when working with the random model of the zeta function) one has perfect orthogonality, so that $\int |P|^4$ can be efficiently bounded. In our setting the situation is quite different -- we have constructed our function $P(t)$ in such a way that the diagonal contribution dominates when calculating $\int |P|^2$, but we do not have orthogonality for $\int |P|^4$ and cannot efficiently bound this integral.\\

{\bf Open Problem 4:} Show that Theorem \ref{th1} remains valid under the weaker assumption $E_N^* \ll N^{3 - \delta}$ for some $\delta>0$. Show that this can be further relaxed to assuming $E_{N,\gamma}^* \ll \gamma N^{4 - \delta}$, for all $\gamma$ in a range from roughly $1/N$ to $1$. It might even be the case that only values of $\gamma$ near a critical size of roughly $1/N$ are relevant. Note that if the condition $E_{N,\gamma}^* \ll \gamma N^{4 - \delta}$ uniformly for $\gamma \in [1/N,1]$ truly is the ``right'' condition, then this would give a unified picture for the real-sequence case as well as for the integer-sequence case. Indeed, in the latter case clearly $E_{N,\gamma}^* = E_N$ for all $\gamma < 1$ and thus the condition would reduce to $E_N \ll N^{3-\delta}$, in accordance with the criterion stated after \eqref{dio}.\\

As noted, in the case of an integer sequence $(x_n)_n$ it is known that $(x_n \alpha)_n$ cannot have Poissonian pair correlation for almost all $\alpha$ when $E_N \gg N^3$. It would be interesting to obtain an analogous result in the case of real sequences.\\

{\bf Open Problem 5:} Show that unlike in the integer case, it is possible for an increasing sequence $(x_n)_{n \geq 1}$ of reals that $E_N^* \gg N^3$ and that $(x_n \alpha)_n$ has Poissonian pair correlation for almost all $\alpha$ (compare Open Problems 1 and 2 above, where $E_N^* \gg N^3$). Establish a criterion (stated for example in terms of $E^*_{N,\gamma}$) which ensures that $(x_n \alpha)_n$ does not have Poissonian pair correlation for almost all $\alpha$. A candidate for such a criterion is that $E_{N,\gamma}^* \gg \gamma N^4$ for some $\gamma = \gamma(N)$ for infinitely many $N$, where maybe one also has to assume that these values of $\gamma$ are of size $\gamma \approx 1/N$.

\section*{Acknowledgements}

CA is supported by the Austrian Science Fund (FWF), projects F-5512, I-3466, I-4945 and Y-901. DE is supported by FWF projects F-5512 and Y-901. MM is supported by FWF project P-33043. We thank Winston Heap, Olivier Robert, Zeev Rudnick, Ilya Shkredov, Igor Shparlinski, Athanasios Sourmelidis and Niclas Technau for discussions and comments.

\bibliographystyle{abbrv}

\end{document}